\documentclass[11pt]{article}
\usepackage{mathrsfs}
\usepackage{amsmath}
\usepackage{bbm}
\usepackage{amsmath,amsthm,amssymb,amscd}
\usepackage{latexsym}
\usepackage{aliascnt}
\usepackage{hyperref}
\usepackage[numbers,sort&compress]{natbib}
\usepackage{CJK}
\allowdisplaybreaks[4]
\usepackage{indentfirst}

\numberwithin{equation}{section}

\newtheorem{definition}{Definition}[section]
\newtheorem{remark}[definition]{Remark}
\newtheorem{theorem}[definition]{Theorem}

\newtheorem{lemma}[definition]{Lemma}
\newtheorem{proposition}[definition]{Proposition}

\newtheorem{example}[definition]{Example}

\begin{document}
\title{{\bf Stability for coupled second order evolution equations with indirect general memory-dampings without the equal-wave-speeds-type hypothesis}
\thanks{The work was supported partly by the National Natural Science Foundation of China (12361049, 12371116, 12171094) and the Shanghai Key Laboratory for Contemporary Applied Mathematics (08DZ2271900).}}

\author{Kun-Peng Jin$^{a}$,  Jin Liang $^{b}$\thanks{Corresponding author. \ E-mail: \ jinliang@sjtu.edu.cn} ,  Ti-Jun Xiao $^{c}$ \\
{\small $^a$ School of Mathematics and Statistics, Nanning Normal University}\\ {\small Nanning 530100, China}\\
{\small $^b$ School of Mathematical Sciences, Shanghai Jiao Tong University}\\ {\small Shanghai 200240, China}\\
{\small $^c$ School of Mathematical Sciences, Fudan University }\\{\small Shanghai 200433, China}}

\date{}
\maketitle
\begin{abstract}
  We study the stability for a system of coupled second order evolution equations with indirect general memory-damping without the equal-wave-speeds-type hypothesis in a Hilbert space, where the damping only appears just in one equation, the memory kernel can  not necessarily be nonnegative and nonincreasing, and the ``wave speeds" implied by the system can be different. Taking advantage of new processing ideas and with the help of the properties of the Generalized Positive Definite Kernel, we overcome difficulties caused by the different ``wave speeds", the lack of the decreasing and nonnegative property for the memory kernel and the system has only one equation with damping, and obtain an optimal polynomial stability result for the energy, which covers the previous related polynomial stability results for second order coupled equations (abstract or concrete) in the literature. Moreover, applications of the abstract result are given.

\vspace{0.1cm}

\vspace{0.1cm}
\end{abstract}

\section{Introduction and preliminaries}

Systems of coupled evolution equations are more realistic  mathematical models describing many phenomena in biology material, viscoelastic fluid, petroleum industry, etc. The well known Timoshenko systems for beam and coupled Petrovsky-type systems are the typical examples for the systems of coupled evolution equations (cf., e.g., \cite{9, 28}). Since the work of Russell \cite{25}, the indirected damping has been considered in the study of stability of various evolution equations. The development of memory material let us see that the research involving the indirected damping via memory effect is necessary and profound. For the coupled systems with the indirected damping via memory, which is associated to two evolution equations and initial-boundary conditions, it is important to control the whole coupled system by only using a single damping for only one single equation. Such a stable problem is different from and much more complex than the case of single equations, as showed in Xiao and Liang \cite{28}, Jin, Liang and Xiao \cite{13,14}.

Let $H$ be a Hilbert space and let $A_{1}$, $A_{2}$, $B_{1}$ and $B_{2}$ be linear operators in $H$. We are concerned with the following abstract Cauchy problem for a system of coupled equations with memory, which arises from the theory of viscoelasticity,
\begin{eqnarray}
&& u''(t)+A_{1}u(t)-\int_{0}^{t}g(t-s)A_{1}u(s)ds+\alpha u(t)+B_{2}v(t)=0,\label{NI}\\[0.15cm]
&& v''(t)+A_{2}v(t)+B_{1}u(t)=0,\label{NII}\\[0.15cm]
& &u(0)=u_0,~~u'(0)=u_1,\label{NIII}\\[0.15cm]
& &v(0)=v_0,~~v'(0)=v_1,\label{NVI}
\end{eqnarray}
where $A_{1}$ and $A_{2}$ are positive self-adjoint operators in $H$, and $g(t)\in L^{1}(0, +\infty)$ is the memory kernel, which produces damping (memory-damping), and $g(t)$ may be oscillating or even sign-varying, that is, $g(t)$ may be not nonnegative and nonincreasing at all.
In the coupled system \eqref{NI}-\eqref{NVI} above, the damping only appears in the first equation, and the second equation is damped indirectly via its coupling with the first one.

In the present paper, we pay our attention to the study the stability for the system \eqref{NI}-\eqref{NVI} without the equal-wave-speeds-type hypothesis.  Basing on the properties of Generalized Positive Definite Kernel (GPDK), which is first introduced by Jin, Liang and Xiao in \cite{JLXGPDK}, and using some new ideas and techniques, we overcome difficulties caused by the different ``wave speeds" implied in the system, the lack of the decreasing and nonnegative property for the memory kernel $g(t)$ and the task of transmitting dissipation from one equation to another, and obtain an optimal polynomial stability result for the energy  due to $G(t)= \int_t^{+\infty} g(s)ds$ being a strongly $(t+1)^{\lambda_{0}}$-positive definite kernel. Our results generalizes and improves the previous related polynomial stability results for second order coupled equations (abstract or concrete) in the literature. For details, please see Theorem \ref{6T2} below. Moreover, we give some applications of the polynomial stability result.

We refer the reader to, e.g., \cite{Ala1, Ala2, Guesmia20252, Can1, Can2, Batty, Cav20, Con1, Con2, De1, De2, 9, G1, Guesmia2020, Guesmia2024, JDE2019, JLXGPDK, M1, Mustafa2022, Mustafa2023, Mustafa20232, P1} and the references cited there, for more results and techniques on the study of second order equations, the Timoshenko-type (Timoshenko) systems and coupled systems, which also stimulate our ideas.

This paper is organized as follows. Section 2 is about the properties of the Generalized Positive Definite Kernel, which will be used to build polynomial stability of the energy.  Section 3 is devoted to deriving the polynomial stability and integrability of the energy, where we give a series of lemmas. In Section 4, applications of our polynomial stability theorem are given.

Throughout the paper, $C$ is a generic positive constant and $H$ is a real Hilbert space with the scalar product $\langle\cdot,\cdot \rangle.$  For a linear operator $A$ in $H,$ $\mathcal{D}(A)$ and $[\mathcal{D}(A)]$ stand for, respectively, its domain and the domain endowed with the graph norm. By $C^i([0,\infty);H_0)$, for a normed linear space $H_0$ and $i\in \{0, 1, 2\}$, we denote the set of all $i$-times continuously differentiable $H_0$-valued functions on
$[0,\infty),$ and $C([0,\infty);H_0):= C^0([0,\infty);H_0)$.

Under following assumptions, we have the global existence and uniqueness theorem, regularity theorem, and approximation theorem below, by using the similar techniques given in \cite{14}.

{\rm \begin{enumerate}
\item[${\rm(I_1)}$]  $A_{1}$, $A_{2}$ are positive self-adjoint linear operators in $H$, with $\mathcal{D}(A_{1})$ and $\mathcal{D}(A_{2})$ dense in $H$, and they satisfy
\begin{eqnarray*}\label{3-1}
        &&a_{1}\langle A_{1}u, u\rangle \geq  \Vert u\Vert ^2 \ \ (u\in \mathcal{D}(A_{1})),\quad
        a_{2}\langle A_{2}v, v\rangle \geq  \Vert v\Vert ^2,\ \ (v\in \mathcal{D}(A_{2})),
\end{eqnarray*}
for constants $a_{1}>0$, $a_{2}>0$.

\item[${\rm(I_2)}$]
$B_{1}$, $B_{2}$ are linear operators in $H$ with $\mathcal{D}(B_{1})\supset\mathcal{D}(\sqrt{A_{1}})$, $\mathcal{D}(B_{2})\supset\mathcal{D}(\sqrt{A_{2}})$, satisfying
\begin{eqnarray*}\label{3-2}
\left\langle B_{1}u, v\right\rangle=\left\langle u,  B_{2}v\right\rangle, \qquad u\in\mathcal{D}(A_1),\quad v\in\mathcal{D}(A_{2}),
\end{eqnarray*}
and there are two positive constants $\beta, \beta_{1}>0$, such that
\begin{eqnarray*}\label{3-3}
 &&\beta\left\|\sqrt{A_{1}}u\right\|\leq\|B_{1}u\|\leq\beta_{1}\left\|\sqrt{A_{1}}u\right\| \quad \mbox{ for }u\in \mathcal{D}(\sqrt{A_{1}}),\\[0.15cm]
 &&\beta\left\|\sqrt{A_{2}}v\right\|\leq\|B_{2}v\|\leq\beta_{1}\left\|\sqrt{A_{2}}v\right\| \quad \mbox{ for }v\in \mathcal{D}(\sqrt{A_{2}}).
\end{eqnarray*}

\item[${\rm(I_3)}$]  $\alpha\geq 0$ is a constant, and $g\in L^{1}(0,+\infty)$ satisfies
\begin{equation*}\label{3-4}
  \int_0^{+\infty}g(t)dt <1,\ \ \int_0^{+\infty}g(t)dt+a_{1}\left(\alpha-\beta_{1}^2\right)<1,\ \
\int_{t}^{+\infty}g(s)ds ~\mbox{is a positive definite kernel}.
\end{equation*}
\end{enumerate}

\vspace{0.6cm}

A pair $(u,v)$ of functions is called a (classical) solution of (\ref{NI})-(\ref{NVI}) on $[0,T)$, $T>0$, if
        \begin{eqnarray*}\label{3-7}
         & &u \in~ C^2\left([0,T);H\right) \cap  C^{1}\left([0,T);[{\mathcal{D}(\sqrt{A_{1}})}]\right) \cap  C\left([0,T);[{\mathcal{D}(A_{1})}]\right),\\[0.15cm]
         & &v \in~ C^2\left([0,T);H\right) \cap  C^{1}\left([0,T);[{\mathcal{D}(\sqrt{A_{2}})}]\right) \cap  C\left([0,T);[{\mathcal{D}(A_{2})}]\right),
\end{eqnarray*}
satisfying \eqref{NI}-\eqref{NVI} for $t\in[0,T)$.

We define the energy of a solution ($u,$
$v$) of \eqref{NI}-\eqref{NVI} as
    \begin{eqnarray*}\label{3-8}
     E(t)=E_{u,v}(t)&=&\frac{1}{2}\left\|u'(t)\right\|^{2}+\frac{1-\int_{0}^{+\infty}g(s)ds}{2}\left\|\sqrt{A_{1}}u(t)\right\|^{2}\\[0.15cm]
     & &+\frac{1}{2}\left\|v'(t)\right\|^{2}+\frac{1}{2}\left\|\sqrt{A_{2}}v(t)\right\|^{2}+\frac{\alpha}{2}\left\|u(t)\right\|^{2}+\left\langle B_{1}u(t), v(t)\right\rangle.
    \end{eqnarray*}

The following theorem is about the global existence and uniqueness of solutions to \eqref{NI}-\eqref{NVI}, and the detailed derivation process can be found in references \cite{14,28}.

\begin{theorem}\label{3T1}
 Let assumptions ${\rm (I_1)}$-${\rm (I_3)}$ hold. Then,
for $u_0\in \mathcal{D}(A_{1})$, $v_0\in \mathcal{D}(A_{2})$,  $u_1\in
\mathcal{D}( \sqrt{A_{1}})$, and $v_1\in
\mathcal{D}( \sqrt{A_{2}})$,
the system \eqref{NI}-\eqref{NVI} has a unique solution  $(u(t), v(t))$ on $[0,+\infty)$, and
its energy $E(t)$ satisfies that for $t\ge 0,$
    \begin{eqnarray}\label{3N-10}
    E(t)\leq CE(0),\quad c_0\widetilde{E}(t)\leq E(t)\leq c_1\widetilde{E}(t),
    \end{eqnarray}
where $$\widetilde{E}(t)=\widetilde{E}_{u,v}(t)=\left\Vert u'(t)\right\Vert^2+\left\Vert v'(t)\right\Vert^2+\left\Vert \sqrt{A_{1}}u(t)\right\Vert^2
        +\left\Vert \sqrt{A_{2}}v(t)\right\Vert^2,$$ $c_0 > 0$, $c_1 > 0$ and $C>0$ are constants.
\end{theorem}

\begin{theorem}\label{reg} Assume that conditions ${\rm (I_1)}$-${\rm (I_3)}$ are satisfied. Let
\begin{eqnarray}\label{d0}
  \mathcal{D}_0=\Big\{(y_1, y_2)\in \mathcal{D}(A_1)\times \mathcal{D}(A_2); \ (A_1+\alpha)y_1+ B_2y_2\in \mathcal{D}(\sqrt{A_1}),~ B_1y_1+A_2y_2\in \mathcal{D}(\sqrt{A_2})\Big\}.
\end{eqnarray}
Then
\begin{enumerate}
\item[{\rm (i)}] The solution $(u(t),v(t))$ of Cauchy problem \eqref{NI}-\eqref{NVI} satisfies
\begin{eqnarray*}
        && u\in~ C^3([0,\infty);H) \cap C^2([0,\infty);[{\mathcal{D}(\sqrt{A_1})}])\cap  C^1([0,\infty);[{\mathcal{D}(A_1)}]),\\[0.15cm]
        && v\in~ C^3([0,\infty);H) \cap C^2([0,\infty);[{\mathcal{D}(\sqrt{A_2})}])\cap  C^1([0,\infty);[{\mathcal{D}(A_2)}]),
\end{eqnarray*}
whenever $(u_0, v_0)\in \mathcal{D}_0,$ and $(u_1, v_1)\in \mathcal{D}(A_1)\times \mathcal{D}(A_2).$
\item[{\rm (ii)}] There exist sequences
\begin{equation*}
  \{(u_{0n},v_{0n})\}_{n=1}^\infty\subset \mathcal{D}_0,\quad \{(u_{1n},v_{1n})\}_{n=1}^\infty\subset \mathcal{D}(A_1)\times \mathcal{D}(A_2)
\end{equation*}
such that the solution $(u_n(t),v_n(t))$ of \eqref{NI}-\eqref{NII} with initial data
\begin{equation*}
   u_n(0)=u_{0n},~v_n(0)=v_{0n},~u_n'(0)=u_{1n},~v_n'(0)=v_{1n}
\end{equation*}
satisfies that as $n\rightarrow\infty$,
\begin{eqnarray*}
&& A_1u_{0n}\rightarrow A_1u_0 ,\quad A_2v_{0n}\rightarrow A_2v_0 ,\label{1}\\[0.18cm]
&& \sqrt{A_1} u_{1n}\rightarrow \sqrt{A_1}u_1 ,\quad \sqrt{A_2}v_{1n}\rightarrow \sqrt{A_2}v_1 ,\label{2}\\[0.18cm]
&& u_n''(t)\rightarrow u''(t) , \quad \sqrt{A_1}u_n'(t)\rightarrow \sqrt{A_1}u'(t) ,\nonumber\\[0.18cm]
 &&  v_n''(t)\rightarrow v''(t) , \quad \sqrt{A_2}v_n'(t)\rightarrow \sqrt{A_2}v'(t)\nonumber
\end{eqnarray*}
in $H$, uniformly on bounded subsets of $t\geq 0$.
\end{enumerate}
\end{theorem}

\section{Generalized Positive Definite Kernel with its properties}

Recall that
\begin{definition}\label{2D2}{\em
Let $h\in L^{1}_{loc}(0,+\infty)$, and $\varphi\in L_{loc}^{\infty}(0, +\infty)$ with $\varphi(t)>0$ for $t>0.$ The function $h$ is called to be a  $\varphi$-positive definite kernel if
\begin{align*}
\int^{t}_{0}\varphi(s)\left\langle h*u(s), u(s)\right\rangle ds\geq 0,   \quad\quad \forall \ t\geq 0,
\end{align*}
for any $u\in L^{2}_{loc}(0,+\infty; H)$. Moreover, $h$ is said to be a strongly $\varphi$-positive definite kernel if there exist two constants
$\delta >0$, $N>0$ such that $h(t)-\delta e^{-Nt}$ is a $\varphi$-positive definite kernel.}
\end{definition}
We call the above function $h$ as the Generalized Positive Definite Kernel (GPDK) (\cite{JLXGPDK}). Next, we recall some properties for the GPDK (see details in \cite{JLXGPDK}).

\begin{proposition}\label{2F-P1}
Suppose that $\kappa$ is a real number. Then $h(t)$  is a (strongly) $e^{\kappa t}$-positive definite kernel if and only if $e^{\frac{\kappa}{2}t}h(t)$ is a (strongly) positive definite kernel.
\end{proposition}

\begin{proposition}\label{2P3}
Assume that $\psi\in L_{loc}^{\infty}(0, +\infty)$ is a decreasing positive function, and $h$ is a (strongly) $\varphi(t)$-positive definite kernel. Then $h$  is a (strongly) $\psi\varphi$-positive definite kernel.
\end{proposition}

\begin{proposition}\label{2NP2}
Assume that $(t+a)^{\nu}h(t)$ $(a\geq 0,$ $\nu\in\mathbb{R^{+}})$ is a (strongly) positive definite kernel. Then, for any $b>0$, $h(t)$ is a (strongly) $(t+a+b)^{\nu}$-positive definite kernel.
\end{proposition}

\begin{proposition}\label{2NP5}
Assume that ~$h(t)\in L^{1}_{loc}(0,\infty)$~ is a strongly $\varphi(t)$-positive definite kernel. Let $\varphi'(t)\leq N\varphi(t)$. Then for ~$t\geq 0$, $u(t)\in L^{2}_{loc}(0,\infty; H)$~ and
~$u'\in L^{1}_{loc}(0,\infty; H)$~, we have
\begin{eqnarray*}\label{2-4}
\nonumber\int_{0}^{t}\varphi(s)\|u(s)\|^{2}ds&\leq & C\|u(0)\|^{2}+\frac{2N}{\delta}\int_{0}^{t}\varphi(s)\left\langle h\ast u(s), u(s)\right\rangle ds\\[0.15cm]
&&+\frac{4}{N\delta}\int_{0}^{t}\varphi(s)\left\langle h\ast u'(s), u'(s)\right\rangle ds,
\end{eqnarray*}
where ~$\delta$~ is the constant in Definition~\autoref{2D2}~.
\end{proposition}
Therefore, the weighted integral of $\|u(s)\|^{2}$ can be controlled by two inner product of a GPDK function $h$ with $u$ and $u'$. This provides us with a good way to estimate the potential energy of the problem \eqref{NI}-\eqref{NVI}.

\begin{proposition}\label{2EX-00}
\begin{enumerate}
\item[{\rm (i)}]
 If $h(t)$ is a (strongly) $e^{bt}$-positive definite kernel, then $h(t)\cos at$ $(a\in \mathbb{R})$ is a (strongly) $e^{bt}$-positive definite kernel.

\item[{\rm (ii)}] If $(t+1)^{b}h(t)$ is a (strongly) positive definite kernel, then $h(t)\cos at$ $(a\in \mathbb{R})$ is a (strongly) $(t+2)^{b}$-positive definite kernel.
\end{enumerate}
\end{proposition}

The following is some examples of GPDK (details see \cite{JLXGPDK}).

\begin{example}\label{2EX1}{\em
\begin{enumerate}
\item[{\rm (i)}] If $g(t)>0$ and $$g'(t)\leq -k_{0}g(t) \quad (k_{0}>0),$$ then $\int_{t}^{+\infty}g(s)ds$ is a strongly $e^{2kt}$-positive definite kernel, where $0\leq k\leq\frac{k_{0}}{2}$.

\item[{\rm (ii)}] If $g(t)>0$ and $$g'(t)\leq -\frac{k_{0}}{t+1}g(t) \quad  (k_{0}>1),$$ then $\int_{t}^{+\infty}g(s)ds$ is a strongly $(t+2)^{\frac{k}{2}}$-positive definite kernel, where $$0<k<\min \{k_{0}-1, \frac{k_{0}}{2}\},$$ and if $k_{0}>2$, $k$ can take $\frac{k_{0}}{2}$.
\end{enumerate}}
\end{example}

\begin{example}\label{NEX1}{\em
\begin{enumerate}
\item[{\rm (i)}] If $h(t)=e^{-bt}\cos at$, then $h(t)$ is a strongly $e^{2ct}$-positive definite kernel, where $0\leq c< b$.

\item[{\rm (ii)}]  If $h(t)=(t+1)^{-b}\cos at$ ($b>0$), then $h(t)$ is a strongly $(t+2)^{c}$-positive definite kernel, where $0\leq c< b$.
\end{enumerate}}
\end{example}

\begin{example}
{\em
\begin{enumerate}
\item[{\rm (i)}] If $$g(t)=be^{-bt}\cos at+ae^{-bt}\sin at,$$ then $\int_{t}^{+\infty}g(s)ds$ is a strongly $e^{2ct}$-positive definite kernel, where $0\leq c< b$.

\item[{\rm (ii)}]  If $$g(t)=b(t+1)^{-b-1}\cos at+a(t+1)^{-b}\sin at \quad (b>0),$$ then $\int_{t}^{+\infty}g(s)ds$ is a strongly $(t+2)^{c}$-positive definite kernel, where $0< c< b$.
\end{enumerate}}
\end{example}

From the definition of GPDK and the above examples, we see the GPDK may be oscillating or sign-varying kernels, which is a significant difference from the kernels studied in the past literature.

\vspace{3mm} \noindent

\section{Polynomial decay property and integrability of the energy for the system \eqref{NI}-\eqref{NVI}}

\subsection{Assumptions and polynomial decay theorem}
Firstly, we strengthen the hypothesis ${\rm(I_3)}$ in Section 3 and append a new assumption ${\rm(I_4)}$.
{\rm
\begin{enumerate}
\item[${\rm(I_3')}$] $g$  satisfies
\begin{equation*}\label{6-1}
  \int_0^{+\infty}g(t)dt <1,\quad \int_0^{+\infty}g(t)dt+a_{1}\left(\alpha-\beta_{1}^2\right)<1,
\end{equation*}
$$(t+1)^{\lambda_{0}}g(t)\in L^{1}(0,+\infty),\quad (t+1)^{\lambda_{0}}g'(s)\in L^{1}(0, +\infty),$$
and
\begin{equation}\label{6pos}
 t\mapsto\int_{t}^{+\infty}g(s)ds  {\rm~is~a~strongly}~ (t+1)^{\lambda_{0}}- {\rm positive~ definite~kernel};
\end{equation}
here $~\lambda_{0}\ge 1$ is a positive constant.

\item[${\rm(I_4)}$] There exists a nonnegative functional $P$ on $B_{1}\mathcal{D}(A_{1})\bigcup B_{2}\mathcal{D}(A_{2}),$ and two bounded linear operators $\Lambda_i$ ($i=1,2$) on $H$ such that for $u\in \mathcal{D}(A_{1})$, $v\in \mathcal{D}(A_{2})$,
    \begin{eqnarray}\label{4-2}
     |\langle A_{1}u, B_{2}v\rangle-\langle B_{1}u, A_{2}v\rangle|\leq D_{1}\left(\left\|B_{1}u\right\|\left\|B_{2}v\right\|+P(B_{1}u)P(B_{2}v)\right),
    \end{eqnarray}
and for $u\in\mathcal{D}(A_{i})$, $i=1,2$
\begin{eqnarray}\label{4-3}
  &&P^{2}(B_{i}u)\leq\left\langle A_{i}u,\Lambda_{i} B_{i}u\right\rangle +D_{2}\left\|B_{i}u\right\|^{2},\label{4-3-1}\\[0.15cm]
  && \left\langle u, \Lambda_{i} B_{i}u\right\rangle\leq D_{2}\|u\|^{2},\label{4-3-2}
\end{eqnarray}
where $D_{1}$ and $D_{2}$ are positive constants.
\end{enumerate}
}

Note that \eqref{6pos} implies, according to Proposition \ref{2P3}, that $\displaystyle t\mapsto\int_{t}^{+\infty}g(s)ds$ is a strongly positive definite kernel. Therefore, we have Theorem \ref{3T1}. Now we give the polynomial decay theorem.
\begin{theorem}\label{6T2}
Assume that ${\rm(I_1)}$, ${\rm(I_2)}$, ${\rm(I^{'}_3)}$ and ${\rm(I_4)}$ hold. Then, for any $u_0\in \mathcal{D}(A_{1})$, ~$v_0\in \mathcal{D}(A_{2})$,  ~$u_1\in
 \mathcal{D}( \sqrt{A_{1}})$ and $v_1\in
 \mathcal{D}( \sqrt{A_{2}})$, the solution energy of \eqref{NI}-\eqref{NVI} satisfies
\begin{eqnarray}
E(t)\leq C(t+1)^{-\lambda_{0}},~~t\ge 0,
\end{eqnarray}
furthermore, we have, for $t\ge 0$,
\begin{eqnarray}
\int_{0}^{+\infty}(t+1)^{\min\{\lambda_{0}, 2(\lambda_{0}-1)\}}E(t)dt&\leq&C.
\end{eqnarray}
Here $C>0$ is a constant.
\end{theorem}
The proof of this polynomial decay theorem is not short. So we first prove some lemmas in the following subsection 3.2 and then present a complete proof of Theorem \ref{6T2} by virtue of these lemmas in subsection 3.3 below.

\subsection{Lemmas and proofs}

The following notations will be used.
\begin{eqnarray*}
G_{\lambda}=\int_{0}^{+\infty}(t+1)^{\lambda}|g(t)|dt, ~~\mbox{for}~ \lambda >0;~~~~  G(t)=\int_{t}^{+\infty}g(s)ds,
~~\mbox{for}~t\ge 0.
\end{eqnarray*}

\begin{lemma}\label{6N-1}
Let ${\rm (I_{3}')}$ hold. Then, for any $0\leq\lambda\leq\lambda_{0}$,
\begin{eqnarray}\label{6-10}
\nonumber& &(t+1)^{\lambda}E(t)+\int^{t}_{0}(s+1)^{\lambda}\left\langle\int^{s}_{0}G(s-\tau)\sqrt{A_{1}}u'(\tau)d\tau, \sqrt{A_{1}}u'(s)\right\rangle ds\\[0.15cm]
&\le&CE(0)+C\lambda\int^{t}_{0}(s+1)^{\lambda-1}E(s)ds, ~~t\ge 0.
\end{eqnarray}
Here $C> 0$ is a constant.
\end{lemma}
\begin{proof}
Multiplying \eqref{NI} and \eqref{NII} by $(t+1)^{\lambda}u'(t)$, $(t+1)^{\lambda}v'(t)$ respectively, we have
\begin{eqnarray}\label{6-4}
\nonumber& &\frac{1}{2}(t+1)^{\lambda}\frac{d}{dt}\left\|u'(t)\right\|^{2}+\frac{1}{2}(t+1)^{\lambda}\frac{d}{dt}\left\|v'(t)\right\|^{2}
+\frac{1}{2}(t+1)^{\lambda}\frac{d}{dt}\left\|\sqrt{A_{1}}u(t)\right\|^{2}\\[0.15cm]
\nonumber& & +\frac{1}{2}(t+1)^{\lambda}\frac{d}{dt}\left\|\sqrt{A_{2}}v(t)\right\|^{2}
+\frac{\alpha}{2}(t+1)^{\lambda}\frac{d}{dt}\left\|u(t)\right\|^{2}+(t+1)^{\lambda}\frac{d}{dt}\left\langle B_{2}v(t), u(t)\right\rangle\\[0.15cm]
&=&(t+1)^{\lambda}\left\langle\int^{t}_{0}g(t-s)\sqrt{A_{1}}u(s)ds, \sqrt{A_{1}}u'(t)\right\rangle.
\end{eqnarray}
Integrating  \eqref{6-4} from $0$ to $t$, we get
\begin{eqnarray}\label{6-5}
\nonumber& &\frac{1}{2}(t+1)^{\lambda}\left\|u'(t)\right\|^{2}+\frac{1}{2}(t+1)^{\lambda}\left\|v'(t)\right\|^{2}
+\frac{1}{2}(t+1)^{\lambda}\left\|\sqrt{A_{1}}u(t)\right\|^{2}\\[0.15cm]
\nonumber&&+\frac{1}{2}(t+1)^{\lambda}\left\|\sqrt{A_{2}}v(t)\right\|^{2}
+\frac{\alpha}{2}(t+1)^{\lambda}\left\|u(t)\right\|^{2}+(t+1)^{\lambda}\left\langle B_{2}v(t), u(t)\right\rangle\\[0.15cm]
\nonumber&\le&C_1E(0)+\int^{t}_{0}(s+1)^{\lambda}\left\langle\int^{s}_{0}g(s-\tau)\sqrt{A_{1}}u(\tau)d\tau, \sqrt{A_{1}}u'(s)\right\rangle ds\\[0.15cm]
\nonumber& &+\frac{1}{2}\int^{t}_{0}\lambda(s+1)^{\lambda-1}\bigg(\left\|u'(s)\right\|^{2}+\left\|v'(s)\right\|^{2}+\left\|\sqrt{A_{1}}u(s)\right\|^{2}
\\[0.15cm]
&&\quad+\left\|\sqrt{A_{2}}v(s)\right\|^{2}+\alpha\left\|u(s)\right\|^{2}+2\left\langle B_{2}v(s), u(s)\right\rangle\bigg)ds.
\end{eqnarray}
Clearly, for $t\ge 0,$
$$
g\ast u(t)=G(0)u(t)-G(t)u(0)-G\ast u'(t).$$
Therefore, we obtain
\begin{eqnarray}\label{6F-1}
\nonumber&&\int^{t}_{0}(s+1)^{\lambda}\left\langle\int^{s}_{0}g(s-\tau)\sqrt{A_{1}}u(\tau)d\tau, \sqrt{A_{1}}u'(s)\right\rangle ds\\[0.15cm]
\nonumber&&+\int^{t}_{0}(s+1)^{\lambda}\left\langle\int^{s}_{0}G(s-\tau)\sqrt{A_{1}}u'(\tau)d\tau, \sqrt{A_{1}}u'(s)\right\rangle ds\\[0.15cm]
\nonumber&=&\frac{G(0)}{2}(t+1)^{\lambda}\left\|\sqrt{A_{1}}u(t)\right\|^{2}+\frac{G(0)}{2}\left\|\sqrt{A_{1}}u(0)\right\|^{2}
\\[0.15cm]
\nonumber&&-G(t)(t+1)^{\lambda}\left\langle\sqrt{A_{1}}u(0), \sqrt{A_{1}}u(t)\right\rangle-\frac{G(0)}{2}\lambda\int_{0}^{t}(s+1)^{\lambda-1}\left\|\sqrt{A_{1}}u(s)\right\|^{2} ds\\[0.15cm]
\nonumber&&-\int_{0}^{t}g(s)(s+1)^{\lambda}\left\langle\sqrt{A_{1}}u(0), \sqrt{A_{1}}u(s)\right\rangle ds\\[0.15cm]
&&+\int_{0}^{t}G(s)\lambda(s+1)^{\lambda-1}\left\langle\sqrt{A_{1}}u(0), \sqrt{A_{1}}u(s)\right\rangle ds.
\end{eqnarray}
Putting \eqref{6F-1} into \eqref{6-5} yields that
\begin{eqnarray}\label{6F-2}
\nonumber& &\frac{1}{2}(t+1)^{\lambda}\bigg(\left\|u'(t)\right\|^{2}+\left\|v'(t)\right\|^{2}
+(1-G(0))\left\|\sqrt{A_{1}}u(t)\right\|^{2}+\left\|\sqrt{A_{2}}v(t)\right\|^{2}
+\alpha\left\|u(t)\right\|^{2}\\[0.15cm]
\nonumber&&+2\left\langle B_{2}v(t), u(t)\right\rangle\bigg)
+\int^{t}_{0}(s+1)^{\lambda}\left\langle\int^{s}_{0}G(s-\tau)\sqrt{A_{1}}u'(\tau)d\tau, \sqrt{A_{1}}u'(s)\right\rangle ds\\[0.15cm]
\nonumber&\le&C_1E(0)+\frac{1}{2}\int^{t}_{0}\lambda(s+1)^{\lambda-1}\bigg(\left\|u'(s)\right\|^{2}+\left\|v'(s)\right\|^{2}+\left\|\sqrt{A_{1}}u(s)\right\|^{2}
+\left\|\sqrt{A_{2}}v(s)\right\|^{2}\\[0.15cm]
\nonumber&&+\alpha\left\|u(s)\right\|^{2}+2\left\langle B_{2}v(s), u(s)\right\rangle\bigg)ds-G(t)(t+1)^{\lambda}\left\langle\sqrt{A_{1}}u(0), \sqrt{A_{1}}u(t)\right\rangle\\[0.15cm]
\nonumber&&-\int_{0}^{t}g(s)(s+1)^{\lambda}\left\langle\sqrt{A_{1}}u(0), \sqrt{A_{1}}u(s)\right\rangle ds\\[0.15cm]
&&+\int_{0}^{t}G(s)\lambda(s+1)^{\lambda-1}\left\langle\sqrt{A_{1}}u(0), \sqrt{A_{1}}u(s)\right\rangle ds.
\end{eqnarray}
Then, by Young inequality and \eqref{6F-2}, we have
\begin{eqnarray}\label{6LP-1}
\nonumber& &(t+1)^{\lambda}E(t)
+\int^{t}_{0}(s+1)^{\lambda}\left\langle\int^{s}_{0}G(s-\tau)\sqrt{A_{1}}u'(\tau)d\tau, \sqrt{A_{1}}u'(s)\right\rangle ds\\[0.15cm]
\nonumber&\le&CE(0)+C\lambda\int^{t}_{0}\lambda(s+1)^{\lambda-1}E(s)ds+\frac{1}{2}|G(t)(t+1)^{\lambda}|\left(\left\|\sqrt{A_{1}}u(0)\right\|^{2}+ \left\|\sqrt{A_{1}}u(t)\right\|^{2}\right)\\[0.15cm]
\nonumber&&+\frac{1}{2}\int_{0}^{t}|g(s)(s+1)^{\lambda}|\left(\left\|\sqrt{A_{1}}u(0)\right\|^{2}+ \left\|\sqrt{A_{1}}u(s)\right\|^{2}\right) ds\\[0.15cm]
&&+\frac{1}{2}\int_{0}^{t}|G(s)\lambda(s+1)^{\lambda-1}|\left(\left\|\sqrt{A_{1}}u(0)\right\|^{2}+ \left\|\sqrt{A_{1}}u(s)\right\|^{2}\right) ds.
\end{eqnarray}
In view of the assumption ${\rm(I'_{3}})$,  it follows that for any $\lambda\leq\lambda_{0}$,
$$\int_{0}^{+\infty}(t+1)^{\lambda}|g(t)|dt\leq \int_{0}^{+\infty}\frac{(t+1)^{\lambda}}{(t+1)^{\lambda_{0}}}(t+1)^{\lambda_{0}}|g(t)|dt\leq \int_{0}^{+\infty}(t+1)^{\lambda_{0}}|g(t)|dt.$$ This implies that $(t+1)^{\lambda}g(t)\in L^{1}(0,+\infty)$ holds for any $\lambda\leq\lambda_{0}$. Similarly, we can deduce that $(t+1)^{\lambda-1}G(t)\in L^{1}(0,+\infty)$ and that $(t+1)^{\lambda}G(t)$ is bounded for $\lambda\leq\lambda_{0}$.

Therefore, by \eqref{6LP-1} and  $$\left\|\sqrt{A_{1}}u(s)\right\|^{2}\leq CE(0),$$ we get the desired estimation \eqref{6-10}.
\end{proof}

By Theorem \ref{reg} and differentiating the system \eqref{NI}-\eqref{NII} respect to $t$, we have
\begin{eqnarray}
&& u'''(t)+A_{1}u'(t)-g(t)A_{1}u(0)-\int_{0}^{t}g(t-s)A_{1}u'(s)ds+\alpha u'(t)+B_{2}v'(t)=0,\qquad\label{DI}\\[0.38cm]
&& v'''(t)+A_{2}v'(t)+B_{1}u'(t)=0.\label{DII}
\end{eqnarray}

Clearly, the above new system \eqref{DI}-\eqref{DII} is similar to the system \eqref{NI}-\eqref{NII}.

Write
\begin{eqnarray}
E_{2}(t):=E(u'(t),v'(t),t),~\widetilde{E}_{2}(t):=\widetilde{E}(u'(t),v'(t),t).
\end{eqnarray}
It follows from assumption ${\rm(I'_3)}$ that there are two constants $c_{0},c_{1}>0$ such that
\begin{eqnarray}\label{EI-1}
0\leq c_{0}\widetilde{E}_{2}(t)\leq E_{2}(t)\leq c_{1}\widetilde{E}_{2}(t),~t\ge 0.
\end{eqnarray}

\begin{lemma}\label{6LD-2}
Let ${\rm (I_{3}')}$ hold. Then for any $0\leq\lambda\leq\lambda_{0}$,
\begin{eqnarray}\label{6-10-1}
\nonumber& &(t+1)^{\lambda}E_{2}(t)+\int^{t}_{0}(s+1)^{\lambda}\left\langle\int^{s}_{0}G(s-\tau)\sqrt{A_{1}}u''(\tau)d\tau, \sqrt{A_{1}}u''(s)\right\rangle ds\\[0.15cm]
&\le&C+C\lambda\int^{t}_{0}(s+1)^{\lambda-1}E_{2}(s)ds, ~~t\ge 0.
\end{eqnarray}
Here $C> 0$ is a constant.
\end{lemma}
\begin{proof}
Taking $u''(t), v''(t)$ as the multipliers of \eqref{DI}, \eqref{DII} respectively and integrating them from $0$ to $t$, we get
\begin{eqnarray*}
\nonumber&&\frac{1}{2}\left\|u''(t)\right\|^{2}+\frac{1}{2}\left\|v''(t)\right\|^{2}+\frac{1}{2}\left\|\sqrt{A_{1}}u'(t)\right\|^{2}+\frac{\alpha}{2}\left\|u'(t)\right\|^{2}+\frac{1}{2}\left\|\sqrt{A_{2}}v'(t)\right\|^{2}+\left\langle B_{1}u'(t), v'(t)\right\rangle\\[0.15cm]
\nonumber&=&\frac{1}{2}\left\|u''(0)\right\|^{2}+\frac{1}{2}\left\|v''(0)\right\|^{2}+\frac{1}{2}\left\|\sqrt{A_{1}}u'(0)\right\|^{2}+\frac{\alpha}{2}\left\|u'(0)\right\|^{2}+\frac{1}{2}\left\|\sqrt{A_{2}}v'(0)\right\|^{2}+\left\langle B_{1}u'(0), v'(0)\right\rangle\\[0.15cm]
&&+\int_{0}^{t}\left\langle g\ast\sqrt{A_{1}}u'(s), \sqrt{A_{1}}u''(s)\right\rangle ds+\int_{0}^{t}\left\langle g(s)A_{1}u(0), u''(s)\right\rangle ds.
\end{eqnarray*}
Thus, the fact that for $t\ge 0,$
$$g\ast u(t)=G(0)u(t)-G(t)u(0)-G\ast u'(t)$$
 and Young's inequality show that
\begin{eqnarray*}
&&E_{2}(t)+\int_{0}^{t}\left\langle G\ast\sqrt{A_{1}}u''(s), \sqrt{A_{1}}u''(s)\right\rangle ds\\[0.15cm]
&\le&C-G(t)\left\langle\sqrt{A_{1}}u'(0), \sqrt{A_{1}}u'(t)\right\rangle-\int_{0}^{t}g(s)\left\langle\sqrt{A_{1}}u'(0), \sqrt{A_{1}}u'(s)\right\rangle ds\\[0.15cm]
&&+\int_{0}^{t}\left\langle g(s)A_{1}u(0), u''(s)\right\rangle ds\\[0.15cm]
&\le&C+\frac{c_{0}}{2}\widetilde{E}_{2}(t)+\int_{0}^{t}|g(s)|\widetilde{E}_{2}(s)ds,
\end{eqnarray*}
Since $G(t)$ is a positive definite kernel and \eqref{EI-1}, we see that
\begin{eqnarray*}
\widetilde{E}_{2}(t)\leq C+C\int_{0}^{t}|g(s)|\widetilde{E}_{2}(s)ds.
\end{eqnarray*}
Therefore, $\widetilde{E}_{2}(t)$ is bounded, that is, there is a constant $C>0$ such that
\begin{eqnarray}\label{DP0}
\left\|u''(t)\right\|^{2}, \left\|v''(t)\right\|^{2}, \left\|\sqrt{A_{1}}u'(t)\right\|^{2}, \left\|\sqrt{A_{2}}v'(t)\right\|^{2}\leq C.
\end{eqnarray}
By the arguments similar to those in the proof of Lemma \ref{6N-1}, multiplying \eqref{DI}-\eqref{DII} by $$(t+1)^{\lambda}u''(t), \quad (t+1)^{\lambda}v''(t)$$ respectively and integrating from $0$ to $t$, we obtain the desired estimation \eqref{6-10-1}.
\end{proof}

\begin{lemma}\label{4LH-1}
Let ${\rm (I_{3}')}$ hold. Then, for any $0\leq\lambda\leq\lambda_{0}$,
\begin{eqnarray}\label{6-12}
\nonumber&&(t+1)^{\lambda}\left(E(t)+E_{2}(t)\right)+\int_{0}^{t}(s+1)^{\lambda}\left\|\sqrt{A_{1}}u'(s)\right\|^{2}ds\\[0.15cm]
\nonumber&\leq & C+C\lambda\int^{t}_{0} (s+1)^{\lambda-1}\left(\left\|\sqrt{A_{1}}u(s)\right\|^{2}
+\left\|\sqrt{A_{2}}v(s)\right\|^{2}\right)ds\\[0.15cm]
&&+C\lambda\int^{t}_{0} (s+1)^{\lambda-1}\left(\left\|u''(s)\right\|^{2}+\left\|v''(s)\right\|^{2}
+\left\|\sqrt{A_{2}}v'(s)\right\|^{2}\right)ds,~~t\ge 0.
\end{eqnarray}
Here $C> 0$ is a constant.
\end{lemma}

\begin{proof}
From \eqref{6-10}, \eqref{6-10-1} and the Proposition \ref{2NP5}, it follows that
\begin{eqnarray}\label{6-11}
\nonumber&&(t+1)^{\lambda}\left(E(t)+E_{2}(t)\right)+\int_{0}^{t}(s+1)^{\lambda}\left\|\sqrt{A_{1}}u'(s)\right\|^{2}ds\\[0.15cm]
&\leq & C+C\lambda\int^{t}_{0}(s+1)^{\lambda-1}\left(E(s)+E_{2}(s)\right)ds.
\end{eqnarray}
Then, by the assumption ${\rm (I_{1})}$ and \eqref{6-11}, we see that
\begin{eqnarray}\label{6-12-1}
\nonumber&&(t+1)^{\lambda}\left(E(t)+E_{2}(t)\right)+\int_{0}^{t}(s+1)^{\lambda}\left\|\sqrt{A_{1}}u'(s)\right\|^{2}ds\\[0.15cm]
\nonumber&\leq & C+C\lambda\int^{t}_{0} (s+1)^{\lambda-1}\left(\left\|\sqrt{A_{1}}u(s)\right\|^{2}
+\left\|\sqrt{A_{2}}v(s)\right\|^{2}\right)ds\\[0.15cm]
\nonumber&&+C\lambda\int^{t}_{0} (s+1)^{\lambda-1}\left(\left\|u''(s)\right\|^{2}+\left\|v''(s)\right\|^{2}
+\left\|\sqrt{A_{2}}v'(s)\right\|^{2}\right)ds\\[0.15cm]
&&+C\lambda\int_{0}^{t}(s+1)^{\lambda-1}\left\|\sqrt{A_{1}}u'(s)\right\|^{2}ds.
\end{eqnarray}
This gives the estimation as required.
\end{proof}

\begin{lemma}\label{6L3}
Let ${\rm (I_{3}')}$ hold. Then, for any $0\leq\lambda\leq\lambda_{0}$,
\begin{eqnarray}\label{6-25}
\nonumber&&\int_{0}^{t}(s+1)^{\lambda}\left\|\sqrt{A_{2}}v'(s)\right\|^{2}ds\\[0.15cm]
\nonumber&\le& C+C(t+1)^{\lambda}\left(E_{2}(t)+\int_{0}^{t}|g(t-s)|\left\|\sqrt{A_{1}}u'(s)\right\|^{2}ds\right)\\[0.15cm]
\nonumber&&+C\int_{0}^{t}(s+1)^{\lambda-1}E_{2}(s)ds+\frac{1}{8}\int_{0}^{t}(s+1)^{\lambda}\left\|v''(s)\right\|^{2}ds\\[0.15cm]
&&+C\int_{0}^{t}(s+1)^{\lambda}\left(\left\|u''(s)\right\|^{2}+\left\|\sqrt{A_{1}}u'(s)\right\|^{2}\right)ds, ~t\ge 0.
\end{eqnarray}
Here $C> 0$ is a constant.
\end{lemma}
\begin{proof}
Define
\begin{eqnarray}\label{6-22}
\nonumber F(t)&:=&\left\langle u''(t),B_{2}v'(t)\right\rangle-\left\langle \left(1-\int_{0}^{t}g(s)ds\right)B_{1}u'(t),v''(t)\right\rangle\\[0.15cm]
&&-\left\langle \int_{0}^{t}g(t-s)\left(B_{1}u'(t)-B_{1}u'(s)\right) ds,v''(t)\right\rangle.
\end{eqnarray}
Then, by \eqref{DI}, \eqref{DII} and assumption ${\rm(I_2)}$, we have
\begin{eqnarray}\label{4-22}
\nonumber\frac{d}{dt}F(t)&=& -\left\|B_{2}v'(t)\right\|^{2}-\alpha \left\langle u'(t),B_{2}v'(t) \right\rangle+g(t)\left\langle A_{1}u(0),B_{2}v'(t)\right\rangle+\left\|B_{1}u'(t)\right\|^{2}\\[0.15cm]
\nonumber&&-\left\langle\int_{0}^{t}g(t-s)B_{1}u'(s)ds, B_{1}u'(t)\right\rangle
+g(t)\left\langle B_{1}u'(t),v''(t)\right\rangle\\[0.15cm]
&&-\left\langle \int_{0}^{t}g'(t-s)(B_{1}u'(t)-B_{1}u'(s)) ds,v''(t)\right\rangle+\Psi,
\end{eqnarray}
where
\begin{eqnarray*}\label{4-23}
\Psi(t):=\left\langle B_{1}u'(t)-\int_{0}^{t}g(t-s)B_{1}u'(s) ds, A_{2}v'(t)\right\rangle-\left\langle A_{1}u'(t)-\int_{0}^{t}g(t-s)A_{1}u'(s) ds, B_{2}v'(t)\right\rangle.
\end{eqnarray*}
Therefore, by \eqref{4-22} and Young's inequality, we deduce that
\begin{eqnarray}\label{F}
\nonumber&&\int_{0}^{t}(s+1)^{\lambda}\left\|B_{2}v'(s)\right\|^{2}ds\\[0.15cm]
\nonumber&\le& |F(0)|+(t+1)^{\lambda}|F(t)|+\lambda\int_{0}^{t}(s+1)^{\lambda-1}|F(s)|ds+\zeta_{1}\int_{0}^{t}(s+1)^{\lambda}\left\|B_{2}v'(s)\right\|^{2}ds\\[0.15cm]
\nonumber&&+\frac{\alpha^{2}}{4\zeta_{1}}\int_{0}^{t}(s+1)^{\lambda}\left\|u'(s)\right\|^{2}ds+\int_{0}^{t}(s+1)^{\lambda}\left\|B_{1}u'(s)\right\|^{2}ds\\[0.15cm]
\nonumber&&+\frac{1}{2}\int_{0}^{t}(s+1)^{\lambda}|g(s)|\left(\left\|A_{1}u(0)\right\|^{2}+\left\|B_{2}v'(s)\right\|^{2}+\left\|B_{1}u'(s)\right\|^{2}+\left\|v''(s)\right\|^{2}\right)ds\\[0.15cm]
\nonumber&&+\frac{1}{2}\int_{0}^{t}(s+1)^{\lambda}\left\|\int_{0}^{s}g(s-\tau)B_{1}u'(\tau)d\tau\right\|^{2}ds+\frac{1}{2}\int_{0}^{t}(s+1)^{\lambda}\left\|B_{1}u'(s)\right\|^{2}ds\\[0.15cm]
\nonumber&&+\frac{1}{4\zeta_{2}}\int_{0}^{t}(s+1)^{\lambda}\left\|\int_{0}^{s}g'(s-\tau)\left(B_{1}u'(s)-B_{1}u'(\tau)\right)d\tau\right\|^{2}ds\\[0.15cm]
&&+\zeta_{2}\int_{0}^{t}(s+1)^{\lambda}\left\|v''(s)\right\|^{2}ds+\int_{0}^{t}(s+1)^{\lambda}|\Psi(s)|ds.
\end{eqnarray}
Since
\begin{eqnarray*}
|F(t)|\le C\left(E_{2}(t)+\int_{0}^{t}|g(t-s)|\left\|B_{1}u'(s)\right\|^{2}ds\right),
\end{eqnarray*}
we see by \eqref{F} that
\begin{eqnarray}\label{F-1}
\nonumber&&\int_{0}^{t}(s+1)^{\lambda}\left\|B_{2}v'(s)\right\|^{2}ds\\[0.15cm]
\nonumber&\le& C+C(t+1)^{\lambda}\left(E_{2}(t)+\int_{0}^{t}|g(t-s)|\left\|B_{1}u'(s)\right\|^{2}ds\right)+C\int_{0}^{t}(s+1)^{\lambda-1}E_{2}(s)ds\\[0.15cm]
\nonumber&&+\zeta_{1}\int_{0}^{t}(s+1)^{\lambda}\left\|B_{2}v'(s)\right\|^{2}ds+\zeta_{2}\int_{0}^{t}(s+1)^{\lambda}\left\|v''(s)\right\|^{2}ds\\[0.15cm]
&&+C\left(\zeta_{1},\zeta_{2}\right)\int_{0}^{t}(s+1)^{\lambda}\left\|B_{1}u'(s)\right\|^{2}ds+\int_{0}^{t}(s+1)^{\lambda}|\Psi(s)|ds.
\end{eqnarray}
On the other hand, by assumption ${\rm(I_4)}$, we have
\begin{eqnarray}\label{PP}
\nonumber\int_{0}^{t}(s+1)^{\lambda}|\Psi(s)|ds&\le&D_{1}\int_{0}^{t}(s+1)^{\lambda}\left\|B_{1}u'(s)-\int_{0}^{s}g(s-\tau)B_{1}u'(\tau) d\tau\right\|\left\|B_{2}v'(s)\right\|ds\\[0.15cm]
\nonumber&&+D_{1}\int_{0}^{t}(s+1)^{\lambda}P\left(B_{1}u'(s)-\int_{0}^{s}g(s-\tau)B_{1}u'(\tau) d\tau\right)P\left(B_{2}v'(s)\right)ds\\[0.15cm]
\nonumber&\le&\frac{D_{1}^{2}}{4\xi_{1}}\int_{0}^{t}(s+1)^{\lambda}\left\|B_{1}u'(s)-\int_{0}^{s}g(s-\tau)B_{1}u'(\tau) d\tau\right\|^{2}ds\\[0.15cm]
\nonumber&&+\xi_{1}\int_{0}^{t}(s+1)^{\lambda}\left\|B_{2}v'(s)\right\|^{2}ds+\frac{D_{1}}{4\xi_{2}}\int_{0}^{t}(s+1)^{\lambda}P^{2}\left(B_{2}v'(s)\right)ds\\[0.15cm]
\nonumber&&+D_{1}\xi_{2}\int_{0}^{t}(s+1)^{\lambda}P^{2}\left(B_{1}u'(s)-\int_{0}^{s}g(s-\tau)B_{1}u'(\tau) d\tau\right)ds\\[0.15cm]
\nonumber&\le&\xi_{1}\int_{0}^{t}(s+1)^{\lambda}\left\|B_{2}v'(s)\right\|^{2}ds+C(\xi_{1})\int_{0}^{t}(s+1)^{\lambda}\left\|B_{1}u'(s)\right\|^{2}ds\\[0.15cm]
\nonumber&&+D_{1}\xi_{2}\int_{0}^{t}(s+1)^{\lambda}P^{2}\left(B_{1}u'(s)-\int_{0}^{s}g(s-\tau)B_{1}u'(\tau) d\tau\right)ds\\[0.15cm]
&&+\frac{D_{1}}{4\xi_{2}}\int_{0}^{t}(s+1)^{\lambda}P^{2}\left(B_{2}v'(s)\right)ds.
\end{eqnarray}
Moreover, by assumption ${\rm(I_4)}$, we obtain
\begin{eqnarray}\label{PU}
\nonumber&&\int_{0}^{t}(s+1)^{\lambda}P^{2}\left(B_{1}u'(s)-\int_{0}^{s}g(s-\tau)B_{1}u'(\tau) d\tau\right)ds\\[0.15cm]
\nonumber&\le&D_{2}\int_{0}^{t}(s+1)^{\lambda}\left\|B_{1}u'(s)-\int_{0}^{s}g(s-\tau)B_{1}u'(\tau) d\tau\right\|^{2}ds\\[0.15cm]
\nonumber&&+\int_{0}^{t}(s+1)^{\lambda}\bigg\langle A_{1}u'(s)-\int_{0}^{s}g(s-\tau)A_{1}u'(\tau) d\tau, \\[0.15cm] \nonumber&&\quad\Lambda_{1}\left(B_{1}u'(s)-\int_{0}^{s}g(s-\tau)B_{1}u'(\tau) d\tau\right)\bigg\rangle ds\\[0.15cm]
\nonumber&\le&C\int_{0}^{t}(s+1)^{\lambda}\left\|B_{1}u'(s)\right\|^{2}ds\\[0.15cm]
\nonumber&&-\int_{0}^{t}(s+1)^{\lambda}\bigg\langle u'''(s), \Lambda_{1}\left(B_{1}u'(s)-\int_{0}^{s}g(s-\tau)B_{1}u'(\tau) d\tau\right)\bigg\rangle ds\\[0.15cm]
\nonumber&&+\int_{0}^{t}(s+1)^{\lambda}\bigg\langle g(s)A_{1}u(0), \Lambda_{1}\left(B_{1}u'(s)-\int_{0}^{s}g(s-\tau)B_{1}u'(\tau) d\tau\right)\bigg\rangle ds\\[0.15cm]
\nonumber&&-\int_{0}^{t}(s+1)^{\lambda}\bigg\langle \alpha u'(s)+B_{2}v'(s), \Lambda_{1}\left(B_{1}u'(s)-\int_{0}^{s}g(s-\tau)B_{1}u'(\tau) d\tau\right)\bigg\rangle ds\\[0.15cm]
\nonumber&\le&C+\xi_{3}\int_{0}^{t}(s+1)^{\lambda}\left\|B_{2}v'(s)\right\|^{2}ds+C(\xi_{3})\int_{0}^{t}(s+1)^{\lambda}\left\|B_{1}u'(s)\right\|^{2}ds\\[0.15cm]
\nonumber&&-(s+1)^{\lambda}\bigg\langle u''(s), \Lambda_{1}\left(B_{1}u'(s)-\int_{0}^{s}g(s-\tau)B_{1}u'(\tau) d\tau\right)\bigg\rangle_{0}^{t}\\[0.15cm]
\nonumber&&+\lambda\int_{0}^{t}(s+1)^{\lambda-1}\bigg\langle u''(s), \Lambda_{1}\left(B_{1}u'(s)-\int_{0}^{s}g(s-\tau)B_{1}u'(\tau) d\tau\right)\bigg\rangle ds\\[0.15cm]
\nonumber&&+\int_{0}^{t}(s+1)^{\lambda}\left\langle u''(s), \Lambda_{1}B_{1}u''(s)\right\rangle ds\\[0.15cm]
\nonumber&&+\int_{0}^{t}(s+1)^{\lambda}\bigg\langle u''(s), \Lambda_{1}\left(\int_{0}^{s}g'(s-\tau)\left(B_{1}u'(s)-B_{1}u'(\tau)\right) d\tau -g(s)B_{1}u'(s)\right)\bigg\rangle ds\\[0.15cm]
\nonumber&\le&C+C(t+1)^{\lambda}\left(\left\|u''(t)\right\|^{2}+\left\|B_{1}u'(t)\right\|^{2}+\int_{0}^{t}|g(t-s)|\left\|B_{1}u'(s)\right\|^{2}ds\right)\\[0.15cm]
&&+\xi_{3}\int_{0}^{t}(s+1)^{\lambda}\left\|B_{2}v'(s)\right\|^{2}ds+C(\xi_{3})\int_{0}^{t}(s+1)^{\lambda}\left(\left\|u''(s)\right\|^{2}+\left\|B_{1}u'(s)\right\|^{2}\right)ds,
\end{eqnarray}
and
\begin{eqnarray}\label{PV}
\nonumber&&\int_{0}^{t}(s+1)^{\lambda}P^{2}\left(B_{2}v'(s)\right)ds\\[0.15cm]
\nonumber&\le&D_{2}\int_{0}^{t}(s+1)^{\lambda}\left\|B_{2}v'(s)\right\|^{2}ds+\int_{0}^{t}(s+1)^{\lambda}\left\langle A_{2}v'(s),\Lambda_{2}B_{2}v'(s)\right\rangle ds\\[0.15cm]
\nonumber&\le&D_{2}\int_{0}^{t}(s+1)^{\lambda}\left\|B_{2}v'(s)\right\|^{2}ds-\int_{0}^{t}(s+1)^{\lambda}\left\langle v'''(s),\Lambda_{2}B_{2}v'(s)\right\rangle ds\\[0.15cm]
\nonumber&&-\int_{0}^{t}(s+1)^{\lambda}\left\langle B_{1}u'(s),\Lambda_{2}B_{2}v'(s)\right\rangle ds\\[0.15cm]
\nonumber&\le&D_{2}\int_{0}^{t}(s+1)^{\lambda}\left\|B_{2}v'(s)\right\|^{2}ds-(s+1)^{\lambda}\left\langle v''(s),\Lambda_{2}B_{2}v'(s)\right\rangle_{0}^{t}\\[0.15cm]
\nonumber&&+\lambda\int_{0}^{t}(s+1)^{\lambda-1}\left\langle v''(s),\Lambda_{2}B_{2}v'(s)\right\rangle ds+\int_{0}^{t}(s+1)^{\lambda}\left\langle v''(s),\Lambda_{2}B_{2}v''(s)\right\rangle ds\\[0.15cm]
\nonumber&&-\int_{0}^{t}(s+1)^{\lambda}\left\langle B_{1}u'(s),\Lambda_{2}B_{2}v'(s)\right\rangle ds\\[0.15cm]
\nonumber&\le&C+C(t+1)^{\lambda}\left(\left\| v''(t)\right\|^{2}+\left\|B_{2}v'(t)\right\|^{2}\right)+\left(D_{2}+\xi_{4}\right)\int_{0}^{t}(s+1)^{\lambda}\left\|B_{2}v'(s)\right\|^{2}ds\\[0.15cm]
\nonumber&&+C(\xi_{4})\int_{0}^{t}(s+1)^{\lambda}\left\|B_{1}u'(s)\right\|^{2}ds+C\int_{0}^{t}(s+1)^{\lambda-1}\left(\left\|v''(s)\right\|^{2}+\left\|B_{2}v'(s)\right\|^{2}\right)ds\\[0.15cm]
&&+D_{2}\int_{0}^{t}(s+1)^{\lambda}\left\|v''(s)\right\|^{2}ds.
\end{eqnarray}
Putting \eqref{PU}, \eqref{PV} into \eqref{PP} gives that
\begin{eqnarray}\label{PP-1}
\nonumber\int_{0}^{t}(s+1)^{\lambda}|\Psi(s)|ds&\le&C+C(t+1)^{\lambda}\left(E_{2}(t)+\int_{0}^{t}|g(t-s)|\left\|B_{1}u'(s)\right\|^{2}ds\right)\\[0.15cm]
\nonumber&&+\left(\xi_{1}+D_{1}\xi_{2}\xi_{3}+\frac{D_{1}}{4\xi_{2}}\left(D_{2}
+\xi_{4}\right)\right)\int_{0}^{t}(s+1)^{\lambda}\left\|B_{2}v'(s)\right\|^{2}ds\\[0.15cm]
\nonumber&&
+\left(\frac{D_{1}}{4\xi_{2}}D_{2}\right)\int_{0}^{t}(s+1)^{\lambda}\left\|v''(s)\right\|^{2}ds\\[0.15cm]
\nonumber&&+C\left(\xi_{1},\xi_{2},\xi_{3},\xi_{4}\right)\int_{0}^{t}(s+1)^{\lambda}\left(\left\|u''(s)\right\|^{2}+\left\|B_{2}u'(s)\right\|^{2}\right)ds\\[0.15cm]
&&+C\int_{0}^{t}(s+1)^{\lambda-1}\left(\left\|v''(s)\right\|^{2}+\left\|B_{2}v'(s)\right\|^{2}\right)ds.
\end{eqnarray}
This, together with \eqref{F-1}, yields that
\begin{eqnarray}\label{F-1-1}
\nonumber&&\int_{0}^{t}(s+1)^{\lambda}\left\|B_{2}v'(s)\right\|^{2}ds\\[0.15cm]
\nonumber&\le& C+C(t+1)^{\lambda}\left(E_{2}(t)+\int_{0}^{t}|g(t-s)|\left\|B_{1}u'(s)\right\|^{2}ds\right)+C\int_{0}^{t}(s+1)^{\lambda-1}E_{2}(s)ds\\[0.15cm]
\nonumber&&+\left(\zeta_{1}+\xi_{1}+D_{1}\xi_{2}\xi_{3}+\frac{D_{1}}{4\xi_{2}}\left(D_{2}+\xi_{4}\right)\right)\int_{0}^{t}(s+1)^{\lambda}\left\|B_{2}v'(s)\right\|^{2}ds\\[0.15cm]
\nonumber&&+\left(\zeta_{2}+\frac{D_{1}}{4\xi_{2}}D_{2}\right)\int_{0}^{t}(s+1)^{\lambda}\left\|v''(s)\right\|^{2}ds\\[0.15cm]
&&+C\left(\zeta_{1},\zeta_{2},\xi_{1},\xi_{2},\xi_{3},\xi_{4}\right)\int_{0}^{t}(s+1)^{\lambda}\left(\left\|u''(s)\right\|^{2}+\left\|B_{2}u'(s)\right\|^{2}\right)ds.
\end{eqnarray}
Take
\begin{eqnarray*}
\zeta_{1}=\frac{1}{8},~~\zeta_{2}=\frac{\beta^{2}}{32},~~\xi_{1}=\frac{1}{8},~~\xi_{2}=\frac{D_{1}D_{2}}{\min\left\{\frac{1}{4}, \frac{\beta^{2}}{8}\right\}},~~\xi_{3}=\frac{1}{8D_{1}\xi_{2}},~~\xi_{4}=D_{2}.
\end{eqnarray*}
Then, by \eqref{F-1-1}, we get
\begin{eqnarray}\label{F-2}
\nonumber\int_{0}^{t}(s+1)^{\lambda}\left\|B_{2}v'(s)\right\|^{2}ds
&\le& C+C(t+1)^{\lambda}\left(E_{2}(t)+\int_{0}^{t}|g(t-s)|\left\|B_{1}u'(s)\right\|^{2}ds\right)\\[0.15cm]
\nonumber&&+C\int_{0}^{t}(s+1)^{\lambda-1}E_{2}(s)ds+\frac{\beta^{2}}{8}\int_{0}^{t}(s+1)^{\lambda}\left\|v''(s)\right\|^{2}ds\\[0.15cm]
&&+C\int_{0}^{t}(s+1)^{\lambda}\left(\left\|u''(s)\right\|^{2}+\left\|B_{2}u'(s)\right\|^{2}\right)ds.
\end{eqnarray}
This, together with assumption ${\rm(I_2)}$, leads the desired conclusion.
\end{proof}

\begin{lemma}\label{4LN-1}
Let ${\rm(I_{3}')}$ hold. Then, for any $$0\le\lambda\leq\min\{\lambda_{0},2\left(\lambda_{0}-1\right)\} \quad (\lambda_{0}\geq 1),$$ we have
\begin{align}\label{6-19}
\nonumber &\int_{0}^{t}(s+1)^{\lambda}\left(\left\|\sqrt{A_{1}}u(s)\right\|^{2}+\left\|\sqrt{A_{2}}v(s)\right\|^{2}\right)ds\\[0.15cm]
\leq&CE(0)+C(t+1)^{\lambda}E(t)+C\int_{0}^{t}(s+1)^{\lambda}\left(\left\|\sqrt{A_{1}}u'(s)\right\|^{2}+\left\|\sqrt{A_{1}}v'(s)\right\|^{2}\right) ds,~~t\ge 0.
\end{align}
Here $C>0$ is a constant.
\end{lemma}
\begin{proof}
By \eqref{NI}-\eqref{NII}, we have
\begin{eqnarray}\label{6-13}
\nonumber&&\frac{d}{dt}\left(\left\langle u'(t), u(t)\right\rangle+\left\langle v'(t), v(t)\right\rangle\right)\\[0.15cm]
\nonumber&=&\left\langle u''(t), u(t)\right\rangle+\left\|u'(t)\right\|^{2}+\left\langle v''(t), v(t)\right\rangle+\left\|v'(t)\right\|^{2}\\[0.15cm]
\nonumber&=&-\bigg(\left\|\sqrt{A_{1}}u(t)\right\|^{2}+\left\|\sqrt{A_{2}}v(t)\right\|^{2}+\alpha\left\|u(t)\right\|^{2}+2\left\langle B_{1}u(t), v(t)\right\rangle\bigg)\\[0.15cm]
\nonumber&&+\left\|u'(t)\right\|^{2}
+\left\|v'(t)\right\|^{2}+\left\langle\int_{0}^{t}g(t-s)\sqrt{A_{1}}u(s)ds, \sqrt{A_{1}}u(t)\right\rangle\\[0.15cm]
\nonumber&=&-\left((1-G(0))\left\|\sqrt{A_{1}}u(t)\right\|^{2}+\left\|\sqrt{A_{2}}v(t)\right\|^{2}+\alpha\left\|u(t)\right\|^{2}+2\left\langle B_{1}u(t), v(t)\right\rangle\right)\\[0.15cm]
 &&+\left\|u'(t)\right\|^{2}+\left\|v'(t)\right\|^{2}-\left\langle G\ast\sqrt{A_{1}}u'(t), \sqrt{A_{1}}u(t)\right\rangle
-G(t)\left\langle \sqrt{A_{1}}u(0), \sqrt{A_{1}}u(t)\right\rangle.
\end{eqnarray}
Moreover, it follows from \cite[Equation (2.13)]{28} that there is $\mu>0$ such that
   \begin{eqnarray}\label{3-10-1}
  \nonumber&&\left(1-G(0)\right)\left\|\sqrt{A_{1}}u(t)\right\|^{2}+\frac{\alpha}{2}\|u(t)\|^{2}+\frac{1}{2} \left\|\sqrt{A_{2}}v(t)\right\|^{2}+\left\langle B_{1}u(t), v(t)\right\rangle\\[0.15cm]
  &\ge&\mu \left(\left\|\sqrt{A_{1}}u(t)\right\|^{2}+\left\|\sqrt{A_{1}}u(t)\right\|^{2}\right).
    \end{eqnarray}
By Young's inequality, \eqref{6-13} and \eqref{3-10-1}, we deduce that
\begin{eqnarray}\label{6-13-1}
\nonumber&&\frac{d}{dt}\left(\left\langle u'(t), u(t)\right\rangle+\left\langle v'(t), v(t)\right\rangle\right)\\[0.15cm]
\nonumber&\leq&-\mu\left(\left\|\sqrt{A_{1}}u(t)\right\|^{2}+\left\|\sqrt{A_{1}}v(t)\right\|^{2}\right)+\left\|u'(t)\right\|^{2}+\left\|v'(t)\right\|^{2}\\[0.15cm]
\nonumber&&+\frac{\mu}{2}\left\|\sqrt{A_{1}}u(t)\right\|^{2}+\frac{1}{\mu}\left\|G\ast\sqrt{A_{1}}u'(t)\right\|^{2}+\frac{G^{2}(t)}{\mu}\left\|\sqrt{A_{1}}u(0)\right\|^{2}\\[0.15cm]
\nonumber&\leq&-\frac{\mu}{2}\left(\left\|\sqrt{A_{1}}u(t)\right\|^{2}+\left\|\sqrt{A_{2}}v(t)\right\|^{2}\right)
+\left\|u'(t)\right\|^{2}+\left\|v'(t)\right\|^{2}\\[0.15cm]
&&+\frac{1}{\mu}\left\|G\ast\sqrt{A_{1}}u'(t)\right\|^{2}+\frac{G^{2}(t)}{\mu}\left\|\sqrt{A_{1}}u(0)\right\|^{2}.
\end{eqnarray}
Therefore, multiplying \eqref{6-13-1} by $(t+1)^{\lambda}$ and integrating it yields that
\begin{eqnarray}\label{6-14}
\nonumber&&\int_{0}^{t}(s+1)^{\lambda}\left(\left\|\sqrt{A_{1}}u(s)\right\|^{2}+\left\|\sqrt{A_{2}}v(s)\right\|^{2}\right)ds\\[0.15cm]
\nonumber&\leq&CE(0)+C(t+1)^{\lambda}E(t)
+C\int_{0}^{t}(s+1)^{\lambda}\left(\left\|u'(s)\right\|^{2}+\left\|v'(s)\right\|^{2}\right)ds\\[0.15cm]
\nonumber&&+C\int_{0}^{t}\lambda(s+1)^{\lambda-1}\left(\left\|u'(s)\right\|^{2}+\left\|v'(s)\right\|^{2}+\left\|\sqrt{A_{1}}u(s)\right\|^{2}+\left\|\sqrt{A_{2}}v(s)\right\|^{2}\right)ds\\[0.15cm]
&&+C\int_{0}^{t}(s+1)^{\lambda}\left\|G\ast\sqrt{A_{1}}u'(s)\right\|^{2}ds+CE(0)\int_{0}^{t}(s+1)^{\lambda}G^{2}(s)ds.
\end{eqnarray}
In addition,
\begin{eqnarray}\label{6-16}
\nonumber\int_{0}^{+\infty}(t+1)^{\lambda_{0}-1}|G(t)|dt&\leq& \int_{0}^{+\infty}\int_{t}^{+\infty}(t+1)^{\lambda_{0}-1}|g(s)|dsdt\\[0.15cm]
\nonumber&\leq&\int_{0}^{+\infty}\int_{0}^{s}(t+1)^{\lambda_{0}-1}|g(s)|dtds\\[0.15cm]
\nonumber&\leq&\frac{1}{\lambda_{0}}\int_{0}^{+\infty}(s+1)^{\lambda_{0}}|g(s)|ds\\[0.15cm]
&\leq&CG_{\lambda_{0}}.
\end{eqnarray}
So, we get
\begin{eqnarray}\label{6-17}
\left\|G\ast\sqrt{A_{1}}u'(s)\right\|^{2}\leq C\int_{0}^{s}(s-\tau+1)^{1-\lambda_{0}}|G(s-\tau)|\left\|\sqrt{A_{1}}u'(\tau)\right\|^{2}d\tau.
\end{eqnarray}
Furthermore, if $\lambda_{0}\geq 1$, then taking $\lambda\leq 2\left(\lambda_{0}-1\right)$, we have
\begin{eqnarray}\label{6-18}
\nonumber&&\int_{0}^{t}(s+1)^{\lambda}\left\|G\ast\sqrt{A_{1}}u'(s)\right\|^{2}ds\\[0.15cm]
\nonumber&\leq & C\int_{0}^{t}\int_{0}^{s}(s-\tau+1)^{\lambda-\lambda_{0}+1}| G(s-\tau)|(\tau+1)^{\lambda}\left\|\sqrt{A_{1}}u'(\tau)\right\|^{2}d\tau ds\\[0.15cm]
\nonumber&\leq&C\int_{0}^{t}\int_{\tau}^{t}(s-\tau+1)^{\lambda_{0}-1}|G(s-\tau)|(\tau+1)^{\lambda}\left\|\sqrt{A_{1}}u'(\tau)\right\|^{2}ds d\tau \\[0.15cm]
&\leq&C\int_{0}^{t}(s+1)^{\lambda}\left\|\sqrt{A_{1}}u'(s)\right\|^{2} ds.
\end{eqnarray}
If $\lambda_{0}\geq 1$, then taking $\lambda\leq 2\lambda_{0}-1$, we know that $(s+1)^{\lambda+1-\lambda_{0}}|G(s)|$ is bounded. Hence
\begin{eqnarray}\label{G}
\nonumber\int_{0}^{+\infty}(s+1)^{\lambda}G^{2}(s)ds&=&\int_{0}^{+\infty}(s+1)^{\lambda+1-\lambda_{0}}|G(s)|(s+1)^{\lambda_{0}-1}|G(s)|ds\\[0.15cm]
&\le&C\int_{0}^{+\infty}(s+1)^{\lambda_{0}-1}|G(s)|ds\le C.
\end{eqnarray}
Thus, combining \eqref{6-14}-\eqref{G}, we get \eqref{6-19}. This concludes the proof.
\end{proof}

\subsection{The proof of Theorem \ref{6T2}}
\begin{proof}[The proof of Theorem \ref{6T2}] We divide the proof into five steps.

$\mathbf{Step}$ $\mathbf{1.}$ It follows from \eqref{DI} that
\begin{eqnarray}\label{U}
\nonumber&&\int_{0}^{t}(s+1)^{\lambda}\left\|u''(s)\right\|^{2}ds\\[0.15cm]
\nonumber&=&\int_{0}^{t}(s+1)^{\lambda}\left\langle u''(s),u''(s)\right\rangle ds\\[0.15cm]
\nonumber&=&(s+1)^{\lambda}\left\langle u''(s),u'(s)\right\rangle_{0}^{t}-\lambda\int_{0}^{t}(s+1)^{\lambda-1}\left\langle u''(s),u'(s)\right\rangle ds\\[0.15cm]
\nonumber&&-\int_{0}^{t}(s+1)^{\lambda}\left\langle u'''(s),u'(s)\right\rangle ds\\[0.15cm]
\nonumber&=&(s+1)^{\lambda}\left\langle u''(s),u'(s)\right\rangle_{0}^{t}-\lambda\int_{0}^{t}(s+1)^{\lambda-1}\left\langle u''(s),u'(s)\right\rangle ds\\[0.15cm]
\nonumber&&+\int_{0}^{t}(s+1)^{\lambda}\left\langle A_{1}u'(s)-g(s)A_{1}(0)-g\ast A_{1}u'(s)+\alpha u'(s)+B_{2}v'(s), u'(s)\right\rangle ds\\[0.15cm]
\nonumber&\le&C+C(t+1)^{\lambda}E_{2}(t)+C\int_{0}^{t}(s+1)^{\lambda-1}E_{2}(s)ds\\[0.15cm]
&&+\varepsilon\int_{0}^{t}(s+1)^{\lambda}\left\|\sqrt{A_{2}}v'(s)\right\|^{2}ds+C(\varepsilon)\int_{0}^{t}(s+1)^{\lambda}\left\|\sqrt{A_{1}}u'(s)\right\|^{2}ds.
\end{eqnarray}
Similarly, it follows from \eqref{DII} that
\begin{eqnarray}\label{V}
\nonumber&&\int_{0}^{t}(s+1)^{\lambda}\left\|v''(s)\right\|^{2}ds\\[0.15cm]
\nonumber&\le&C+C(t+1)^{\lambda}E_{2}(t)+C\int_{0}^{t}(s+1)^{\lambda-1}E_{2}(s)ds\\[0.15cm]
&&+\frac{5}{4}\int_{0}^{t}(s+1)^{\lambda}\left\|\sqrt{A_{2}}v'(s)\right\|^{2}ds+C\int_{0}^{t}(s+1)^{\lambda}\left\|\sqrt{A_{1}}u'(s)\right\|^{2}ds.
\end{eqnarray}
$\mathbf{Step}$ $\mathbf{2.}$ By \eqref{U}, \eqref{V} and Lemma \ref{6L3}, we have
\begin{eqnarray*}
\nonumber&&\int_{0}^{t}(s+1)^{\lambda}\left(k\left\|u''(s)\right\|^{2}+\left\|v''(s)\right\|^{2}+2\left\|\sqrt{A_{2}}v'(s)\right\|^{2}\right)ds\\[0.15cm]
\nonumber&\le&C(k)+C(k)(t+1)^{\lambda}\left(E_{2}(t)+\int_{0}^{t}|g(t-s)|\left\|\sqrt{A_{1}}u'(s)\right\|^{2}ds\right)\\[0.15cm]
\nonumber&&+C(k)\int_{0}^{t}(s+1)^{\lambda-1}E_{2}(s)ds+\left(\frac{5}{4}+k\varepsilon\right)\int_{0}^{t}(s+1)^{\lambda}\left\|\sqrt{A_{2}}v'(s)\right\|^{2}ds\\[0.15cm]
\nonumber&&+C(k,\varepsilon)\int_{0}^{t}(s+1)^{\lambda}\left\|\sqrt{A_{1}}u'(s)\right\|^{2}ds\\[0.15cm]
&&+\frac{1}{4}\int_{0}^{t}(s+1)^{\lambda}\left\|v''(s)\right\|^{2}ds+\overline{C}\int_{0}^{t}(s+1)^{\lambda}\left\|u''(s)\right\|^{2}ds.
\end{eqnarray*}
Therefore, taking $k=\overline{C}+1,~\varepsilon=\frac{1}{4k}$ yields that
\begin{eqnarray*}
\nonumber&&\int_{0}^{t}(s+1)^{\lambda}\left(\left\|u''(s)\right\|^{2}+\left\|v''(s)\right\|^{2}+\left\|\sqrt{A_{2}}v'(s)\right\|^{2}\right)ds\\[0.15cm]
\nonumber&\le&C+C(t+1)^{\lambda}\left(E_{2}(t)+\int_{0}^{t}|g(t-s)|\left\|\sqrt{A_{1}}u'(s)\right\|^{2}ds\right)\\[0.15cm]
\nonumber&&+C\int_{0}^{t}(s+1)^{\lambda-1}E_{2}(s)ds+C\int_{0}^{t}(s+1)^{\lambda}\left\|\sqrt{A_{1}}u'(s)\right\|^{2}ds\\[0.15cm]
\nonumber&\leq& C+C(t+1)^{\lambda}E_{2}(t)+C(t+1)^{\lambda}\int_{0}^{t}|g(t-s)|\left\|\sqrt{A_{1}}u'(s)\right\|^{2}ds\\[0.15cm]
\nonumber&&+C\int_{0}^{t}(s+1)^{\lambda-1}\left(\left\|u''(s)\right\|^{2}+\left\|v''(s)\right\|^{2}+\left\|\sqrt{A_{2}}v'(s)\right\|^{2}\right)ds\\[0.15cm]
&&+C\int_{0}^{t}(s+1)^{\lambda}\left\|\sqrt{A_{1}}u'(s)\right\|^{2}ds.
\end{eqnarray*}
Thus, we get, for any $\lambda\leq\lambda_{0}$,
\begin{eqnarray}\label{6-31}
\nonumber&&\int_{0}^{t}(s+1)^{\lambda}\left(\left\|u''(s)\right\|^{2}+\left\|v''(s)\right\|^{2}+\left\|\sqrt{A_{2}}v'(s)\right\|^{2}\right)ds\\[0.15cm]
\nonumber&&\leq C+C(t+1)^{\lambda}E_{2}(t)+C\int_{0}^{t}(s+1)^{\lambda}\left\|\sqrt{A_{1}}u'(s)\right\|^{2}ds\\[0.15cm]
&&+C(t+1)^{\lambda}\int_{0}^{t}|g(t-s)|\left\|\sqrt{A_{1}}u'(s)\right\|^{2}ds.
\end{eqnarray}
Moreover, we have,  for any $\lambda\leq\lambda_{0}$,
\begin{eqnarray}\label{6-31-1}
\nonumber&&\int_{0}^{t}(s+1)^{\lambda-1}\left(\left\|u''(s)\right\|^{2}+\left\|v''(s)\right\|^{2}+\left\|\sqrt{A_{2}}v'(s)\right\|^{2}\right)ds\\[0.15cm]
\nonumber&&\leq C+C(t+1)^{\lambda-1}E_{2}(t)+C\int_{0}^{t}(s+1)^{\lambda-1}\left\|\sqrt{A_{1}}u'(s)\right\|^{2}ds\\[0.15cm]
&&+C(t+1)^{\lambda-1}\int_{0}^{t}|g(t-s)|\left\|\sqrt{A_{1}}u'(s)\right\|^{2}ds.
\end{eqnarray}

$\mathbf{Step}$ $\mathbf{3.}$  If $\lambda_{0}\geq 1$,  $$\min\{\lambda_{0},2\lambda_{0}-1\}=\lambda_{0}.$$ Then,  by \eqref{6-19}, we have, for any $\lambda\leq\lambda_{0}$ ($\lambda_{0}\geq 1$),
\begin{eqnarray}\label{6-19-1}
\nonumber&&\int_{0}^{t}(s+1)^{\lambda-1}\left(\left\|\sqrt{A_{1}}u(s)\right\|^{2}+\left\|\sqrt{A_{2}}v(s)\right\|^{2}\right)ds\\[0.15cm]
\nonumber&\leq&CE(0)+C(t+1)^{\lambda-1}E(t)\\[0.15cm]
&&+C\int_{0}^{t}(s+1)^{\lambda-1}\left(\left\|\sqrt{A_{1}}u'(s)\right\|^{2}+\left\|\sqrt{A_{1}}v'(s)\right\|^{2}\right) ds,~~t\ge 0.
\end{eqnarray}
$\mathbf{Step}$ $\mathbf{4.}$  Putting \eqref{6-31-1}, \eqref{6-19-1} into \eqref{6-12} gives that, if $\lambda_{0}\geq 1$,  for any $\lambda\leq \lambda_{0}$,
\begin{eqnarray}\label{6-32}
\nonumber&&(t+1)^{\lambda}\left(E(t)+E_{2}(t)\right)+\int_{0}^{t}(s+1)^{\lambda}\left\|\sqrt{A_{1}}u'(s)\right\|^{2}ds\\[0.15cm]
\nonumber&\leq & C+C(t+1)^{\lambda-1}\left(E(t)+E_{2}(t)\right)+C\int_{0}^{t}(s+1)^{\lambda-1}\left\|\sqrt{A_{1}}u'(s)\right\|^{2}ds\\[0.15cm]
&&+C(t+1)^{\lambda-1}\int_{0}^{t}|g(t-s)|E_{2}(s)ds.
\end{eqnarray}
That is, if $\lambda_{0}\geq 1$,  for any $\lambda\leq \lambda_{0}$,
\begin{eqnarray}\label{6-32}
\nonumber&&(t+1)^{\lambda}\left(E(t)+E_{2}(t)\right)+\int_{0}^{t}(s+1)^{\lambda}\left\|\sqrt{A_{1}}u'(s)\right\|^{2}ds\\[0.15cm]
\nonumber&\leq & C+C(t+1)^{\lambda-1}\int_{0}^{t}|g(t-s)|E_{2}(s)ds\\[0.15cm]
&\leq & C+C\int_{0}^{t}(t-s+1)^{\lambda-1}|g(t-s)|(s+1)^{\lambda-1}E_{2}(s)ds.
\end{eqnarray}
Therefore, by mathematical induction, we have, for  any $\lambda\leq \lambda_{0}$($\lambda_{0}\geq 1$), $t\geq 0$,
\begin{eqnarray}\label{6-32-11}
(t+1)^{\lambda}\left(E(t)+E_{2}(t)\right)+\int_{0}^{t}(s+1)^{\lambda}\left\|\sqrt{A_{1}}u'(s)\right\|^{2}(s)ds&\leq & C.
\end{eqnarray}
Therefore, \eqref{6-32-11} implies that for any $\lambda_{0}\geq 1$,
\begin{eqnarray}
E(t)\leq C(t+1)^{-\lambda_{0}}, \qquad t\geq 0.
\end{eqnarray}
$\mathbf{Step}$ $\mathbf{5.}$
Meanwhile, it follows from \eqref{6-32-11} and \eqref{6-31} that  for any $\lambda\leq \lambda_{0}$($\lambda_{0}\geq 1$),
\begin{eqnarray}\label{6-32-2}
\int_{0}^{t}(s+1)^{\lambda}E_{2}(s)ds\leq  C.
\end{eqnarray}
Hence, by \eqref{6-19}, we see that if $\lambda_{0}\geq 1$, then for any $$\lambda\leq \min\{\lambda_{0}, 2(\lambda_{0}-1)\},$$
\begin{eqnarray}\label{6-36}
\int_{0}^{t}(s+1)^{\lambda}\left(\left\|\sqrt{A_{1}}u(s)\right\|^{2}+\left\|\sqrt{A_{2}}v(s)\right\|^{2}\right)ds\leq C.
\end{eqnarray}
Finally, by virtue of \eqref{6-32-2} and \eqref{6-36},  we get the desired conclusion. The proof is complete.
\end{proof}

\begin{remark}\label{5NR1}{\em
If the assumptions in ${\rm(I'_{3})}$ $$(t+1)^{\lambda_{0}}g(t),~(t+1)^{\lambda_{0}}g'(t) \in L^{1}(0,+\infty),$$
$$\int_{t}^{+\infty}g(s)ds ~{\rm~is~a~ strongly}~ (t+1)^{\lambda_{0}}-{\rm positive~ definite~kernel},$$
are replaced by that, for any constant $a>0$,
$$(t+a)^{\lambda_{0}}g(t),~(t+a)^{\lambda_{0}}g'(t) \in L^{1}(0,+\infty),$$
$$\int_{t}^{+\infty}g(s)ds ~~is~a~ strongly~ (t+a)^{\lambda_{0}}-positive~ definite~kernel,$$
then the Theorem \ref{6T2} still holds.}
\end{remark}

\section{Applications}
In this section, we present two examples showing how to apply our abstract results to specific
problems.

\begin{example}\label{AE1}{\em
We consider the following Timoshenko system, which describes the transverse vibration of a beam of length $1$:
\begin{align}
&\rho_{1}\varphi_{tt}-\left(b_{1}(x)\varphi_{x}+q\psi\right)_{x}=0, ~~(x,t)\in(0,1)\times(0,\infty), \label{8-1}\\[0.15cm]
&\rho_{2}\psi_{tt}-\left(b_{2}(x)\psi_{x}-g\ast b_{2}(x)\psi_{x}\right)_{x} +q\varphi_{x}+\alpha_{1}\psi=0,~~(x,t)\in(0,1)\times(0,\infty),\label{8-2}
\end{align}
with the initial data
\begin{eqnarray}\label{8-3}
\varphi(\cdot,0)=\varphi_{0},\quad \varphi_{t}(\cdot,0)=\varphi_{1},\quad \psi(\cdot,0)=\psi_{0},\quad \psi_{1}(\cdot,0)=\psi_{1}, \quad x\in (0,1),
\end{eqnarray}
subject to the boundary conditions
\begin{align}\label{8-5}
\left\{\begin{aligned}
&\sigma_{11}\varphi(0,t)-\tau_{11}\varphi_{x}(0,t)=0,~~\sigma_{12}\varphi(1,t)+\tau_{12}\varphi_{x}(1,t)=0,\\[0.15cm]
&\sigma_{21}\psi(0,t)-\tau_{21}\psi_{x}(0,t)=0,~~\sigma_{22}\psi(1,t)+\tau_{22}\psi_{x}(1,t)=0,
\end{aligned}\right.
\end{align}
where $\sigma_{ij}, \tau_{ij}$ are nonnegative constants such that, for $i,j=1,2$,
\begin{eqnarray*}
\sigma_{ij}+\tau_{ij}>0,~~\sigma_{i1}+\sigma_{i2}>0,~~\tau_{1j}\tau_{2j}=0.
\end{eqnarray*}
Here, the unknown functions $\varphi$ and $\psi$ describe, respectively, the transverse displacement of the beam and the rotation angle of a filament,
 $\rho_{1}$, $\rho_{2}$, $q$, and $\alpha_{1}$ are positive constants and $0<b_{1}(x), b_{2}(x)\in C^{1}\left([0,1]\right)$ satisfying
 \begin{align*}
 \rho_{1}^{-1}b_{1}(x)=\rho_{2}^{-1}b_{2}(x),~~x\in[0,1]; \quad
   \alpha_{1}\ge \frac{q^{2}}{\min b_{1}}.
 \end{align*}
Let
\begin{eqnarray*}
u=\psi,\quad v=\sqrt{\frac{\rho_{1}}{\rho_{2}}}\varphi,\quad \alpha=\frac{\alpha_{1}}{\rho_{2}},\quad a(x)=\frac{b_{1}(x)}{\rho_{1}}.
\end{eqnarray*}
Take $H=L^{2}(0,l)$,
\begin{eqnarray*}
A_{1}=-\frac{d}{dx}\left(a\frac{d}{dx}.\right),\quad A_{2}=-\frac{d}{dx}\left(a\frac{d}{dx}.\right),\quad B_{2}=-B_{1}=\frac{q}{\sqrt{\rho_{1}\rho_{2}}}\frac{d}{dx},
\end{eqnarray*}
with
\begin{eqnarray*}
&&\mathcal{D}(A_{1})= \left\{y \in H^{2}(0,1); ~~\sigma_{21}y(0)-\tau_{21}y'(0)=\sigma_{22}y(1)+\tau_{22}y'(1)=0\right\},\\[0.15cm]
&&\mathcal{D}(A_{2})= \left\{y \in H^{2}(0,1); ~~\sigma_{11}y(0)-\tau_{11}y'(0)=\sigma_{12}y(1)+\tau_{12}y'(1)=0\right\},\\[0.15cm]
&&\mathcal{D}(B_{1})=\mathcal{D}(B_{2})=H^{1}(0,1).
\end{eqnarray*}
Then, from \cite[Example 4.1]{28}, the system \eqref{8-1}-\eqref{8-5} can be transformed into the abstract Cauchy problem \eqref{NI}-\eqref{NVI}, and  assumptions ${\rm (I_{1})}$, ${\rm (I_{2})}$  and ${\rm (I_{4})}$ are satisfied.

Therefore, if $g(t)$ satisfies ${\rm (I'_{3})}$, then we conclude that the solution energy of the system \eqref{8-1}-\eqref{8-5} satisfies
\begin{eqnarray*}
E(t)\leq C(t+1)^{-\lambda_{0}}, \quad t\ge 0.
\end{eqnarray*}
Here, $C$ is a positive constant.}
\end{example}

\begin{example}{\em
Consider the system of coupled Petrovsky type equations as follows:
\begin{equation*}\left\{\begin{aligned}
&\partial_{t}^{2}u(t,\xi)+\Delta^{2}u(t,\xi)-\int_0^tg(t-s)\Delta^{2}u(s,\xi)ds+\alpha u(t,\xi)-\beta\Delta v(t,\xi)=0,&t\ge 0, \xi\in \Omega, \\[0.15cm]
&\partial_{t}^{2}v(t,\xi)+\Delta^{2}v(t,\xi)-\beta\Delta u(t,\xi)=0,&t\ge 0, \xi\in \Omega,\\[0.15cm]
&u(t,\xi)=v(t,\xi)=\Delta u(t,\xi)=\Delta v(t,\xi)=0,&t\ge 0, \xi\in \partial\Omega,\\[0.15cm]
&u(0,\xi)=u_{0}(\xi), v(0,\xi)=v_{0}(\xi), \partial_{t}u(0,\xi)=u_{1}(\xi), \partial_{t}v(0,\xi)=v_{1}(\xi),&\xi\in \Omega,
\end{aligned}
\right.\end{equation*}
where $\Omega$ is a bounded domain in $\mathbb{R}^{3}$, with smooth boundary $\partial\Omega$, $\alpha\ge 0$, $\beta>0$. Take $H=L^{2}(\Omega)$,
\begin{eqnarray*}
& &B=-\Delta,~~\mathcal{D}(B)=H^{2}(\Omega)\cap H^{1}_{0}(\Omega),\\[0.15cm]
& &A_{1}=A_{2}=B^{2},~~B_{1}=B_{2}=B.
\end{eqnarray*}
Clearly, assumptions ${\rm (I_{1})}$,  ${\rm (I_{2})}$ and ${\rm (I_{4})}$  are satisfied. Thus, we can utilize our Theorem \ref{6T2} to obtain various rates of energy decay based on the properties of the memory kernel functions $g$.}
\end{example}

\vspace{0.5cm}

{\bf Acknowledgements}

The work was supported partly by the National Natural Science Foundation of China (12361049, 12371116, 12171094) and the Shanghai Key Laboratory for Contemporary Applied Mathematics (08DZ2271900).

{\bf Contributions}

The authors contributed equally and significantly in writing this paper.

{\bf Data Availability}

Data sharing is not applicable to this article as no datasets were generated or analyzed during the current study.

{\bf Conflict of interest}

The authors have no competing interests to declare.

{\bf Ethical Statement}

There are no ethical concerns applicable to our research.

}
\end{document}